\documentclass[preprint,11pt,sort&compress,numbers,draft]
{elsarticle}

\usepackage{amsmath,amsthm,amsfonts,amssymb,latexsym,mathrsfs,color,cases,url,enumerate}
\usepackage{tikz}

\newtheorem{thm}{Theorem}[section]
\newtheorem{lem}[thm]{Lemma}
\newtheorem{coro}[thm]{Corollary}

\newtheorem*{SE}{Schoenberg-Edrei Theorem}
\newtheorem*{ASW}{Aissen-Schoenberg-Whitney Theorem}
\newtheorem*{RM}{Fundamental Theorem of Recursive Matrices}

\theoremstyle{definition}

\newtheorem{exm}[thm]{Example}

\newtheorem{rem}[thm]{Remark}

\newcommand{\lrf}[1]{\lfloor #1\rfloor}

\def\sp{\preccurlyeq}
\def\ssp{\prec}
\def\b{\beta}
\def\c{\gamma}
\def\d{\delta}
\def\R{\mathcal{R}}

\def\sgn{\mathrm{sgn\,}}

\def\n{NICE}
\def\Tt{\mathscr{T}(b,c;x)}
\def\Mx{\mathscr{M}(b,c;x)}
\def\Gx{\mathscr{T}_n(x)}
\def\Gxz{\mathscr{T}_{n-1}(x)}
\def\Gxf{\mathscr{T}_{n+1}(x)}
\def\Gxs{\mathscr{T}_{n+2}(x)}
\def\gx{\mathscr{T}_n}
\def\gxz{\mathscr{T}_{n-1}}
\def\gxf{\mathscr{T}_{n+1}}
\def\AX{\mathscr{A}_n}
\def\Tbc{T^+(b,c)}
\def\Mbc{M(b,c)}
\def\TU{T}
\def\RM{R}
\def\HM{H}
\def\DM{D}
\def\PM{P_a}
\def\JM{J}
\def\AM{A}
\def\BM{B}
\def\LM{L}

\numberwithin{equation}{section}

\journal{JMAA}

\usepackage{geometry}
\begin{document}

\begin{frontmatter}

\title{Analytic aspects of generalized central trinomial coefficients}
\author[a,b,c]{Huyile Liang\corref{cor1}}
\ead{lianghuyile@imnu.edu.cn}
\author[d]{Yaling Wang\corref{cor2}}
\ead{wang-yaling@hotmail.com}
\author[d]{Yi Wang\corref{cor3}}
\ead{wangyi@dlut.edu.cn}
\cortext[cor3]{Corresponding author.}
\address[a]{College of Mathematics Science,
Inner Mongolia Normal University, Hohhot 010022, P.R. China}
\address[b]{Center for Applied Mathematics, Inner Mongolia, Hohhot 010022, P.R. China}
\address[c]{Key Laboratory of Infinite-dimensional Hamiltonian System and Its Algorithm Application,
Ministry of Education, Hohhot 010022, P.R. China}
\address[d]{School of Mathematical Sciences, Dalian University of Technology, Dalian 116024, P.R. China}
\date{}

\begin{abstract}
The divisibility and congruence of
usual and generalized central trinomial coefficients
have been extensively investigated.
The present paper is devoted to analytic properties of these numbers.
We show that usual central trinomial polynomials $\Gx$ have only real roots,
and roots of $\Gx$ interlace those of $\Gxf$,
as well as those of $\Gxs$,
which gives an affirmative answer to an open question of Fisk.
We establish necessary and sufficient conditions such that the generalized central trinomial coefficients $T_n(b,c)$
form a log-convex sequence or a Stieltjes moment sequence.
\end{abstract}


\end{frontmatter}

\section{Introduction}

The central trinomial coefficient $T_n$
is the coefficient of $x^n$ in the expansion $(x^2+x+1)^n$.
By the trinomial theorem,
it is clear that
$T_n=\sum_{k=0}^{\lrf{n/2}}T(n,k)$,
where
\begin{equation}\label{t-nk}
 T(n,k)=\frac{n!}{k!k!(n-2k)!}=\binom{n}{2k}\binom{2k}{k}.
\end{equation}
Although the study of central trinomial coefficients can be traced back to Euler's work,
there are occasional references to them in combinatorics books until Andrews \cite{And90}.
In recent years,
there has been a lot of research work devoted to the study of the central trinomial coefficients
and their generalizations.
For example,
the generalized central trinomial coefficient $T_n(b,c)$
is defined as the coefficient of $x^n$ in the expansion $(x^2+bx+c)^n$.
Clearly, $T_n(1,1)=T_n$.
It is known \cite{GKP,Noe06,Wilf} that
\begin{equation}\label{tnbc}
 T_n(b,c)=\sum_{k=0}^{\lrf{n/2}}T(n,k)b^{n-2k}c^k.
\end{equation}
Also, the generalized central trinomial coefficients have the generating function
\begin{equation}\label{tnbc-gf}
\Tt=\sum_{n\ge 0}T_n(b,c)x^n=\frac{1}{\sqrt{1-2bx+(b^2-4c)x^2}},
\end{equation}
and satisfy the recurrence relation
\begin{equation}\label{tnbc-rr}
(n+1)T_{n+1}(b,c)=(2n+1)bT_n(b,c)-n(b^2-4c)T_{n-1}(b,c)
\end{equation}
with $T_0(b,c)=1$ and $T_1(b,c)=b$.

The generalized central trinomial coefficients
are generalizations of many well-known combinatorial numbers \cite{Noe06}.
For example,
$T_n(2,1)$ is the central binomial coefficient $\binom{2n}{n}$ and
$T_n(3,2)$ is the central Delannoy number $D_n=\sum_{k=0}^n\binom{n}{k}\binom{n+k}{k}$.
On the other hand,
$T_n(1,x)$ is the central trinomial polynomial,
$T_n(x+1,x)$ is the Narayana polynomial
of type B,
$T_n(2x+1,x(x+1))$ is the Delannoy polynomial
and $T_n(x,(x^2-1)/4)$ is the Legendre polynomial.

There have been quite a few papers concerned with the divisibility and congruence of
generalized central trinomial coefficients \cite{CW22,Noe06,Sun14,Sun22}.
The objective of the present paper is to investigate analytic properties of
usual and generalized central trinomial coefficients.

The paper is organized as follows.
In the next section,
we show that the central trinomial polynomials $\Gx=\sum_{k=0}^{\lrf{n/2}}T(n,k)x^k$
have only real roots,
and roots of $\Gx$ interlace those of $\Gxf$,
as well as those of $\Gxs$,
which gives an affirmative answer to an open question of Fisk.
We also show that the matrix $\TU=[T(n,k)]_{n,k\ge 0}$ is totally positive
and the numbers $T(n,k)$ are asymptotically normal (by central and local limit theorems).
In \S 3,
we investigate the generalized central trinomial coefficients $T_n(b,c)$
by setting them in a broader context (the leftmost column of the matrix $\Tbc$).
We show that the matrix $\Tbc$ is both a Riordan array and a recursive matrix of Aigner.
As applications,
we give the binomial and Hankel transforms of $(T_n(b,c))_{n\ge 0}$,
and establish necessary and sufficient conditions such that $(T_n(b,c))_{n\ge 0}$ form a log-convex sequence or a Stieltjes moment sequence.
Furthermore, we show that
$T_{n-1}(b,c)T_{n+1}(b,c)-T_{n}^2(b,c)$
are polynomials in $c$ and $b^2-2c$
with nonnegative integer coefficients for $n\ge 1$.
Finally in \S 4,
we point out that the generalized Motzkin numbers are closely related to the generalized central trinomial coefficients
and they possess many similar properties.

\section{Usual central trinomial coefficients}

In this section, we investigate analytic properties of the usual central trinomial coefficients $T_n$.
Define the matrix $\TU=[T(n,k)]_{n,k\ge 0}$.
Then
\begin{equation}\label{tnk-m}
\TU=
\left(
  \begin{array}{cccc}
    1 &  &  & \\
    1 &  &  & \\
    1 & 2 &  & \\
    1 & 6 &  & \\
    1 & 12 & 6 & \\
    1 & 20 & 30 & \\
    \vdots &  &  &\ddots\\
  \end{array}
\right).
\end{equation}
Let $\Gx=\sum_{k=0}^{\lrf{n/2}}T(n,k)x^k$ be the $n$th row generating function of $\TU$.
Then $\deg \Gx=\lrf{n/2}$,
$\Gx=T_n(1,x)$ and $T_n=\sum_kT(n,k)=\gx(1)$.
It is clear from \eqref{tnbc-rr} that
\begin{equation}\label{tx-rr}
(n+1)\Gxf=(2n+1)\Gx+n(4x-1)\Gxz
\end{equation}
and
\begin{equation}\label{t-rr}
  (n+1)T_{n+1}=(2n+1)T_n+3nT_{n-1}.
\end{equation}

Let $f(x)$ be a {\it real-rooted} polynomial,
i.e., a real polynomial with only real roots.
Denote by $r_i(f)$ the roots of $f$ sorted in non-increasing order:
$r_1(f)\ge r_2(f)\ge r_3(f)\ge\cdots$.
Let $f,g$ be two real-rooted polynomials and $\deg g\le\deg f\le\deg g+1$.
We say that $g$ {\it interlaces} $f$, denoted by $g\sp f$,
if
\begin{equation}\label{int}
r_{1}(f)\ge r_{1}(g)\ge r_{2}(f)\ge r_2(g)\ge\cdots.
\end{equation}
If all inequalities in \eqref{int} are strict, then we say that $g$ {\it strictly interlaces} $f$
and denote it by $g\ssp f$.
For notational convenience,
we say that
a real number $a$ is real-rooted and $a\ssp bx+c$ for $a,b+c>0$ and $b,c\ge 0$.

It is known that $\Gx$ is real-rooted (see \cite[Corollary 13.15]{Fisk} for instance).
Fisk posed a question that roots of $\Gx$ interlace those of $\Gxf$,
as well as those of $\Gxs$
\cite[Question 4, P. 721]{Fisk}.
We next give an affirmative answer of this question.
For completeness, we also include a proof of the real-rootedness of $\Gx$.

\begin{thm}\label{Trr}
The polynomials $\Gx=\sum_{k=0}^{\lrf{n/2}}T(n,k)x^k$
are real-rooted.
Furthermore,
$\Gxz\ssp \Gxf$ for $n\ge 1$ and $\Gx\ssp \Gxf$ for $n\ge 0$.
\end{thm}
\begin{proof}
We first show that all polynomials $\Gx$ are real-rooted and $\Gx\ssp \Gxf$.
We proceed by induction on $n$.
The case is clear for $n\le 2$.
Assume that $\Gxz,\Gx$ are real-rooted and $\Gxz\ssp \Gx$.
We need to show that $\Gxf$ is real-rooted and $\Gx\ssp \Gxf$.

Denote by $\sgn (t)$ the sign of a real number $t$, i.e.,
$$\sgn(t)=
\left\{
  \begin{array}{rl}
    +1, & \hbox{for $t>0$;} \\
    0, & \hbox{for $t=0$;} \\
    -1, & \hbox{for $t<0$.}
  \end{array}
\right.$$
Let $r_{1}^{(n)}>r_{2}^{(n)}>r_{3}^{(n)}>\cdots$ be roots of $\Gx$.
By $\Gxz\ssp \Gx$, we have
$$\sgn \gxz(r_i^{(n)})=(-1)^{i-1},\quad i=1,2,\ldots,\lrf{n/2}.$$
By substituting $x=r_i^{(n)}$ into the equation \eqref{tx-rr}, we obtain
$$\sgn \gxf(r_i^{(n)})=-\sgn \gxz(r_i^{(n)})=(-1)^{i},\quad i=1,2,\ldots,\lrf{n/2}.$$
It follows that $\Gxf$ is real-rooted and $\Gx\ssp \Gxf$,
as required.

We then show that $\Gxz\ssp \Gxf$.
It suffices to consider the case $n\ge 3$.
Now $\Gxz\ssp \Gx$, which implies that
$$\sgn \gx(r_i^{(n-1)})=(-1)^i,\quad i=1,2,\ldots,\lrf{(n-1)/2}.$$
Again by \eqref{tx-rr},
$$\sgn \gxf(r_i^{(n-1)})=\sgn \gx(r_i^{(n-1)})=(-1)^{i},\quad i=1,2,\ldots,\lrf{(n-1)/2}.$$
Thus $\Gxz\ssp \Gxf$.
This completes the proof of the theorem.
\end{proof}

By the same argument as in the proof of Theorem \ref{Trr},
many well-known combinatorial polynomials can be shown to be real-rooted,
including
the Narayana polynomials NB$_n(x)=T_n(x+1,x)$ of type B,
the Delannoy polynomials $D_n(x)=T_n(2x+1,x(x+1))$,
and the Legendre polynomials $P_n(x)=T_n(x,(x^2-1)/4)$.
We refer the reader to \cite{CWZ20,LPW23,LW-rz,WZC19}.

The real-rootedness of polynomials with nonnegative coefficients
is closely related to the total positivity of matrices \cite{FJ11,Fisk,Pin10,Kar68}.
Following Karlin \cite{Kar68},
we say that a (finite or infinite) matrix $\AM$ is {\it totally positive of order $r$}
(TP$_r$ for short),
if its minors of all orders $\le r$ are nonnegative.
The matrix is called {\it totally positive} (TP for short) if its minors of all orders are nonnegative.
Let $(a_k)_{k\ge 0}$ be a (finite or infinite) sequence of nonnegative numbers
(we identify a finite sequence $a_0,a_1,\ldots,a_n$ with the infinite sequence $a_0,a_1,\ldots,a_n,0,0,\ldots$).
We say that the sequence is a {\it P\'olya frequency} (PF for short) sequence
if the corresponding infinite Toeplitz matrix
$$[a_{n-k}]_{n,k\ge 0}
  =\left(
  \begin{array}{ccccc}
    a_0 & 0 & 0 & 0 & \cdots \\
    a_1 & a_0 & 0 & 0 &  \\
    a_2 & a_1 & a_0 & 0 &  \\
    a_3 & a_2 & a_1 & a_0 &  \\
    \vdots &  &  &  & \ddots \\
  \end{array}
\right)$$
is TP.
The following is a fundamental characterization for PF sequences (see~\cite[p. 412]{Kar68} for instance).

\begin{SE}
A sequence $(a_k)_{k\ge 0}$ of nonnegative numbers is PF if and only if its generating function has the form
\begin{equation*}\label{SE-fps}
  \sum_{k\ge 0}a_kx^k=ax^me^{\gamma x}\frac{\prod_{j\ge0}(1+\alpha_j x)}{\prod_{j\ge0} (1-\beta_j x)},
\end{equation*}
where $a>0, m\in\mathbb{N}, \alpha_j, \beta_j, \gamma\ge 0$
and $\sum_{j\ge0} (\alpha_j+\beta_j)<+\infty$.
\end{SE}

The following special case of Schoenberg-Edrei Theorem
establishes the link between finite PF sequences and real-rooted polynomials.

\begin{ASW}
A finite sequence $(a_0,a_1,\ldots,a_n)$ of nonnegative numbers
is PF if and only if its generating function $a_0+a_1x+\cdots+a_nx^n$ is real-rooted.
\end{ASW}

\begin{thm}
The matrix $\TU=[T(n,k)]_{n,k\ge 0}$ is totally positive.
\end{thm}
\begin{proof}
Clearly,
the sequence $\left(1/k!\right)_{k\ge 0}$ is PF
by Schoenberg-Edrei Theorem.
Thus the the corresponding Toeplitz matrix $\left[1/(n-k)!\right]_{n,k\ge 0}$ is TP,
and so is the submatrix $\left[1/(n-2k)!\right]_{n,k\ge 0}$ consisting of even columns of this Toeplitz matrix.
It is obvious that
if the matrix $[a_{n,k}]_{n,k\ge 0}$ is TP, two sequences $b_n$ and $c_k$ are positive,
then the matrix $[a_{n,k}b_nc_k]_{n,k\ge 0}$ is still TP.
The matrix
$$\TU=\left[\frac{n!}{k!k!(n-2k)!}\right]_{n,k\ge 0}$$
is therefore TP.
\end{proof}

We next consider the asymptotic normality of the numbers $T(n,k)$.
Let $a(n,k)$ be a double-indexed sequence of nonnegative numbers and let
\begin{equation*}\label{pnk}
p(n,k)=\frac{a(n,k)}{\sum_{j=0}^na(n,j)}
\end{equation*}
denote the probabilities.
Following Bender~\cite{Ben73},
we say that the sequence $a(n,k)$ is {\it asymptotically normal by a central limit theorem},
if
\begin{equation}\label{clt}
\lim_{n\rightarrow\infty}\sup_{x\in\mathbb{R}}\left|\sum_{k\le\mu_n+x\sigma_n}p(n,k)-\frac{1}{\sqrt{2\pi}}\int_{-\infty}^xe^{-t^2/2}dt\right|=0,
\end{equation}
where $\mu_n$ and $\sigma^2_n$ are the mean and variance of $p(n,k)$, respectively.
We say that $a(n,k)$ is {\it asymptotically normal by a local limit theorem} on $\mathbb{R}$ if
\begin{equation}\label{llt}
\lim_{n\rightarrow\infty}\sup_{x\in\mathbb{R}}\left|\sigma_np(n,\lfloor\mu_n+x\sigma_n\rfloor)-\frac{1}{\sqrt{2\pi}}e^{-x^2/2}\right|=0.
\end{equation}
In this case,
\begin{equation}\label{asy}
a(n,k)\sim \frac{e^{-x^2/2}\sum_{j=0}^na(n,j)}{\sigma_n\sqrt{2\pi}} \textrm{ as } n\rightarrow \infty,
\end{equation}
where $k=\mu_n+x\sigma_n$ and $x=O(1)$.
Clearly, the validity of \eqref{llt} implies that of \eqref{clt}.

Many well-known combinatorial sequences enjoy central and local limit theorems.
For example, the famous de Moivre-Laplace theorem states that
the binomial coefficients $\binom{n}{k}$
are asymptotically normal (by central and local limit theorems).
Other examples include
the signless Stirling numbers $c(n,k)$ of the first kind,
the Stirling numbers $S(n,k)$ of the second kind,
and the Eulerian numbers $A(n,k)$(see \cite{Can15} for instance).
A standard approach to demonstrating asymptotic normality is the following criterion
(see \cite[Theorem 2]{Ben73} for instance).

\begin{lem}\label{lem-rzv}
Suppose that $\AX(x)=\sum_{k=0}^na(n,k)x^k$ have only real roots,
where all $a(n,k)$ are nonnegative.
Let
\begin{equation}\label{mu-f}
\mu_n=\frac{\AX'(1)}{\AX(1)}
\end{equation}
and
\begin{equation}\label{sg-f}
\sigma^2_n=\frac{\AX''(1)}{\AX(1)}+\mu_n-\mu_n^2.
\end{equation}
Then if $\sigma_n^2\rightarrow+\infty$,
the numbers $a(n,k)$ are asymptotically normal (by central and local limit theorems)
with the mean $\mu_n$ and variance $\sigma_n^2$.
\end{lem}

\begin{thm}\label{an}
The numbers $T(n,k)$ are asymptotically normal.
\end{thm}
\begin{proof}
It suffices to show that $\sigma_n^2\rightarrow+\infty$ by Lemma \ref{lem-rzv}.
Recall that
$$\Gx=\sum_{k=0}^{\lrf{n/2}}\binom{n}{2k}\binom{2k}{k}x^k
$$
and $\gx(1)=T_n$.
We have
\begin{eqnarray*}
\gx'(1)
&=&\sum_{k}k\binom{n}{2k}\binom{2k}{k}=\dfrac{n}{2}\sum_{k}\binom{n-1}{2k-1}\binom{2k}{k}\\
&=&\dfrac{n}{2}\sum_{k}\left[\binom{n}{2k}-\binom{n-1}{2k}\right]\binom{2k}{k}\\
&=&\dfrac{n}{2}(T_n-T_{n-1}).
\end{eqnarray*}
It is known \cite[Corollary 2.6]{WZ14} that
$T_{n-1} T_{n+1}\ge T_{n}^2$ for $n\ge 4$.
The sequence $T_{n-1}/T_{n}$ is therefore non-decreasing and convergent.
Furthermore, $T_{n-1}/T_{n}\rightarrow 1/3$ by the difference equation \eqref{t-rr}.
It follows that
\begin{equation*}\label{mu}
\mu_n=\frac{\gx'(1)}{\gx(1)}=\frac{n(T_n-T_{n-1})}{2T_n}
=n\left(\frac{1}{3}+\mathrm{o}(1)\right).
\end{equation*}
On the other hand,
\begin{equation}\label{2+1}
\gx''(1)+\gx'(1)
=\sum_{k}k^2\binom{n}{2k}\binom{2k}{k}
=\sum_k\frac{n!}{(k-1)!(k-1)!(n-2k)!}
=n(n-1)T_{n-2}.
\end{equation}
By \eqref{sg-f},
we have
$$\sigma_n^2=\frac{\gx''(1)+\gx'(1)}{\gx(1)}-\left[\frac{n(T_n-T_{n-1})}{2T_n}\right]^2
=\frac{n[4(n-1)T_{n-2}T_n-nT_n^2+2nT_{n-1}T_n-nT_{n-1}^2]}{4T_n^2},$$
and so
$$\sigma_n^2\ge\frac{n[(3n-4)T_{n-2}T_n-nT_n^2+2nT_{n-1}T_n]}{4T_n^2}
=\frac{n[(3n-4)T_{n-2}-nT_n+2nT_{n-1}]}{4T_n}
=\frac{n(T_{n-1}-T_{n-2})}{4T_n}$$
by the inequality $T_{n-1}^2\le T_{n-2}T_n$
and the difference equation $nT_n=(2n-1)T_{n-1}+3(n-1)T_{n-2}$.
Note that
$$\frac{n(T_{n-1}-T_{n-2})}{4T_n}
=\frac{n}{4}\cdot\frac{T_{n-1}}{T_n}\left(1-\frac{T_{n-2}}{T_{n-1}}\right)
=\frac{n}{4}\left(\frac{2}{9}+o(1)\right)
=\frac{n}{18}+o(n).$$
Thus $\sigma_n^2\rightarrow+\infty$, as required.
\end{proof}

\section{Generalized central trinomial coefficients}

Let
\begin{equation}\label{tbc-1}
  \left(x+b+\frac{c}{x}\right)^n=\sum_{k=-n}^nT_{n,k}(b,c)x^k.
\end{equation}
Then by the trinomial theorem, we have
\begin{equation}\label{tnk}
T_{n,k}(b,c)
=\sum_j\frac{n!}{j!(j+k)!(n-k-2j)!}b^{n-k-2j}c^j.
\end{equation}
In particular,
the constant term $T_{n,0}(b,c)$ is precisely the generalized central trinomial coefficient $T_n(b,c)$,
and $T_{n,-k}(b,c)=c^kT_{n,k}(b,c)$.

Define an infinity lower triangular matrix $\Tbc=[T_{n,k}(b,c)]_{n,k\ge 0}$.
Then
\begin{equation}\label{tbc-tri}
\Tbc=
\left(
  \begin{array}{llllll}
    1 &  &  &  &  &  \\
    b & 1 &  &  &  &  \\
    b^2+2c & 2b & 1 &  &  &  \\
    b^3+6bc & 3b^2+3c & 3b & 1 &  &  \\
    b^4+12b^2c+6c^2 & 4b^3+12bc & 6b^2+4c & 4b & 1 &  \\
    \vdots&\cdots&\cdots&&&\ddots
  \end{array}
\right).
\end{equation}
In this section,
we first show that $\Tbc$ is both a recursive matrix and a Riordan array,
and then apply the total positivity of matrices
to study analytic properties of generalized central trinomial coefficients $T_n(b,c)$.

Let $\b=(b_k)_{k\ge 0}$ and $\c=(c_k)_{k\ge 1}$ be two sequences of nonnegative numbers
and define an infinite lower triangular matrix $\RM=[r_{n,k}]_{n,k\ge 0}$
by the recurrence relation
\begin{equation}\label{rst-eq}
r_{0,0}=1,\qquad r_{n+1,k}=r_{n,k-1}+b_kr_{n,k}+c_{k+1}r_{n,k+1},
\end{equation}
where $r_{n,k}=0$ unless $n\ge k\ge 0$.
Following Aigner~\cite{Aig99,Aig01},
we say that $\RM$ is the {\it recursive matrix}
and $r_{n,0}$ are the {\it Catalan-like numbers}
corresponding to $(\b,\c)$ respectively.
The entry $r_{n,k}$ in $\RM$ counts weighted Motzkin lattice paths.
The following result is known as the fundamental theorem of recursive matrices
\cite[Fundamental Theorem]{Aig01}.

\begin{lem}\label{ft-rm}
Let $\d_0=1$ and $\d_n=c_1c_2\cdots c_n$. Then
\begin{equation}\label{ft1}
  \sum_{k}r_{m,k}r_{n,k}\d_k=r_{m+n,0}.
\end{equation}
Or equivalently,
the Hankel matrix $\HM=[r_{i+j,0}]_{i,j\ge 0}$ of Catalan-like numbers $r_{n,0}$
has the decomposition
\begin{equation}\label{snf}
  \HM=\RM \DM \RM^t,
\end{equation}
where
$\DM=\mathrm{diag}(\d_0,\d_1,\d_2,\ldots)$ is a diagonal matrix and $\RM^t$ is the transpose of the matrix $\RM$.
\end{lem}

The following corollary follows immediately from \eqref{snf}.
\begin{coro}\label{hd-cln}
The $n$th Hankel determinant of the sequence $(r_{k,0})_{k\ge 0}$ is given by
$$\det [r_{i+j,0}]_{0\le i,j\le n}=\d_0\d_1\cdots \d_n.$$
\end{coro}

Let $f(x)=\sum_{n\ge 0}f_nx^n$ and $g(x)=\sum_{n\ge 0}g_nx^n$
be two formal power series.
A {\it Riordan array},
denoted by $\R(g(x),f(x))$,
is an infinite matrix whose generating function of the $k$th column is $x^kf^k(x)g(x)$ for $k\ge 0$.
We say that a Riordan array $\R(g(x),f(x))$ is {\it proper}
if $g_0=1$ and $f_0\neq 0$.
In this case,
$\R(g(x),f(x))$ is an infinite lower triangular matrix.
Riordan arrays play an important unifying role in enumerative combinatorics,
especially in dealing with combinatorial sums
\cite{SGWW91,Spr94}.
The following result is known as the fundamental theorem of Riordan arrays by Shapiro \cite{Sha94}. 

\begin{lem}\label{ft-ra}
Let $\RM=\R(g(x),f(x))=[r_{n,k}]_{n,k\ge 0}$ be a Riordan array
and $h(x)=\sum_{k\ge 0}h_kx^k$.
Then
\begin{equation}\label{ftra-f}
 \sum_{k=0}^nr_{n,k}h_k=[x^n]g(x)h(xf(x)).
\end{equation}
\end{lem}

A proper Riordan array $\RM=[r_{n,k}]_{n,k\ge 0}$
can be characterized by two sequences
$(a_k)_{k\ge 0}$ and $(z_k)_{k\ge 0}$ such that
\begin{equation}\label{rrr-c}
r_{0,0}=1,\quad r_{n+1,0}=\sum_{j\ge 0}z_jr_{n,j},\quad r_{n+1,k+1}=\sum_{j\ge 0}a_jr_{n,k+j}
\end{equation}
for $n,k\ge 0$.
Let $A(x)=\sum_{n\ge 0}a_nx^n$ and $Z(x)=\sum_{n\ge 0}z_nx^n$.
Then
\begin{equation}\label{gf-z}
g(x)=\frac{1}{1-xZ(xf(x))}
\end{equation}
and
\begin{equation}\label{gf-a}
f(x)=A(xf(x)).
\end{equation}
We refer the reader to \cite{HS09} for details.

\begin{exm}
Let $\b=(b,b,b,\ldots)$ and $\c=(c,c,c,\ldots)$.
Then the corresponding Catalan-like numbers are the generalized Motzkin numbers
\begin{equation}\label{gmn}
  M_n(b,c)=\sum_{k=0}^{\lrf{n/2}}\frac{n!}{k!(k+1)!(n-2k)!}b^{n-2k}c^k
  =\sum_{k=0}^{\lrf{n/2}}\binom{n}{2k}\frac{1}{k+1}\binom{2k}{k}b^{n-2k}c^k.
\end{equation}
Let $\Mbc$ denote the corresponding recursive matrix.
Then $\Mbc$ is also a Riordan array
with $A(x)=1+bx+cx^2$ and $Z(x)=b+cx$.
Furthermore,
$\Mbc=\R(\Mx,\Mx)$,
where
\begin{equation}\label{mbct}
  \Mx=\sum_{n\ge 0}M_n(b,c)x^n
  =\frac{1-bx-\sqrt{1-2bx+(b^2-4c)x^2}}{2cx^2}.
\end{equation}
See \cite{WZ15} for details.
\end{exm}

\begin{thm}\label{tbc-ra}
\begin{enumerate}[\rm (i)]
  \item
  The triangle $\Tbc$ is the recursive matrix corresponding to
  $\b=(b,b,b,\ldots)$ and $\c=(2c,c,c,\ldots)$
  and the generalized central trinomial coefficients $T_n(b,c)$
  are the associated Catalan-like numbers.
  \item
  The triangle $\Tbc$ is a Riordan array and $\Tbc=\R(\Tt, \Mx)$.
\end{enumerate}
\end{thm}
\begin{proof}
(i)\quad
By the definition \eqref{tbc-1}, we have $T_{0,0}(b,c)=1$,
\begin{equation}\label{tbc-d0}
T_{n+1,0}(b,c)=T_{n,-1}(b,c)+bT_{n,0}(b,c)+cT_{n,1}(b,c)=bT_{n,0}(b,c)+2cT_{n,1}(b,c)
\end{equation}
since $T_{n,-1}(b,c)=cT_{n,1}(b,c)$, and
\begin{equation}\label{tbc-d1}
T_{n+1,k+1}(b,c)=T_{n,k}(b,c)+bT_{n,k+1}(b,c)+cT_{n,k+2}(b,c)
\end{equation}
for $n,k\ge 0$.
We conclude that
$\Tbc=[T_{n,k}(b,c)]_{n,k\ge 0}$ is the recursive matrix
corresponding to
$\b=(b,b,b,\ldots)$ and $\c=(2c,c,c,\ldots)$
and the generalized trinomial coefficients $T_n(b,c)=T_{n,0}(b,c)$
are the associated Catalan-like numbers.

(ii)\quad
Clearly, the triangle $\Tbc$ is a Riordan array with $A(x)=1+bx+cx^2$ and $Z(x)=b+2cx$.
Let $\Tbc=\R(g(x),f(x))$.
Then by \eqref{gf-a},
we have $f(x)=1+bxf(x)+cx^2f^2(x)$,
and so
\begin{equation}\label{h-e}
cx^2f^2(x)-(1-bx)f(x)+1=0.
\end{equation}
Solve \eqref{h-e} to obtain
$$f(x)=\frac{1-bx-\sqrt{(1-bx)^2-4cx^2}}{2cx^2}
=\Mx.$$
By \eqref{gf-z}, we have
$$g(x)=\frac{1}{1-x[b+2cx\cdot f(x)]}
=\frac{1}{\sqrt{(1-bx)^2-4cx^2}}
=\frac{1}{\sqrt{1-2bx+(b^2-4c)x^2}}=\Tt.$$
Thus $\Tbc=\R(g(x),f(x))=\R(\Tt, \Mx)$.
\end{proof}

The following result has occurred in \cite[Theorem 4]{PRB11},
and is clear from Corollary \ref{hd-cln}.

\begin{coro}\label{tbc-hd}
The $n$th Hankel determinant of the sequence $(T_{k}(b,c))_{k\ge 0}$ is given by
$$\det [T_{i+j}(b,c)]_{0\le i,j\le n}=2^nc^{\binom{n+1}{2}}.$$
\end{coro}

\begin{thm}\label{tbc-bf}
We have
\begin{equation}\label{tbc-bt}
\sum_{k=0}^n\binom{n}{k}T_k(b,c)a^{n-k}=T_n(a+b,c).
\end{equation}
\end{thm}
\begin{proof}
Let
$$\PM=\left[\binom{n}{k}a^{n-k}\right]_{n,k\ge 0}$$
be the generalized Pascal triangle.
Then the generating function of the $k$th column of $\PM$ is
$$\sum_{n\ge 0} \binom{n}{k}a^{n-k} x^n=\frac{x^k}{(1-ax)^{k+1}}.$$
Thus $\PM$ is a Riordan array:
$$\PM=\R\left(\frac{1}{1-ax},\frac{1}{1-ax}\right).$$
By \eqref{ftra-f} and \eqref{tnbc-gf},
we have
\begin{eqnarray*}
  \sum_{k=0}^n\binom{n}{k}a^{n-k}\cdot T_k(b,c)
  &=& [x^n]\frac{1}{1-ax}\frac{1}{\sqrt{(1-b\frac{x}{1-ax})^2-\frac{4cx^2}{(1-ax)^2}}}\\
  &=& [x^n]\frac{1}{\sqrt{[1-(a+b)x]^2-4cx^2}}\\
  &=& T_n(a+b,c).
\end{eqnarray*}
This proves \eqref{tbc-bt}.
\end{proof}

The {\it Hankel transform} $(y_n)_{n\ge 0}$ and the {\it binomial transform} $(z_n)_{n\ge 0}$
of a sequence $(x_n)_{n\ge 0}$ are defined by
$y_n=\det [x_{i+j}]_{0\le i,j\le n}$ and
$z_n=\sum_{k=0}^n\binom{n}{k}x_k$ respectively.
We obtain the Hankel transform of the sequence $(T_n(b,c)_{n\ge 0})$ from Corollary \ref{tbc-hd}.
Also, we can obtain the binomial transform $\sum_{k=0}^n\binom{n}{k}T_k(b,c)=T_n(b+1,c)$
by taking $a=1$ in Theorem \ref{tbc-bf}.

Let $\alpha=(a_k)_{k\ge0}$ be an infinite sequence of nonnegative numbers.
We say that the sequence is {\it log-convex} if $a_ia_{j+1}\ge a_{i+1}a_j$ for $0\le i< j$.
If there exists a nonnegative Borel measure $\mu$ on $[0,+\infty)$ such that
\begin{equation}\label{h-i-e}
a_k=\int_{0}^{+\infty}x^kd\mu(x),
\end{equation}
then we say that $(a_k)_{k\ge 0}$ is a {\it Stieltjes moment} (SM for short) sequence.
Let $\HM(\alpha)=[a_{i+j}]_{i,j\ge 0}$ be the {\it Hankel matrix} of the sequence $\alpha$.
Then
$$\HM(\alpha)
=\left[\begin{array}{lllll}
a_{0} & a_1 & a_2 & a_3 & \cdots\\
a_{1} & a_{2} & a_3 & a_4 & \cdots\\
a_{2} & a_{3} & a_{4} & a_5 & \cdots\\
a_{3} & a_{4} & a_{5} & a_{6} & \cdots\\
\vdots & \vdots &\vdots & \vdots & \ddots\\
\end{array}\right].$$
Clearly, $\alpha$ is log-convex if and only if $\HM(\alpha)$ is TP$_2$.
It is well known that the sequence $\alpha$ is SM if and only if
the corresponding Hankel matrix $\HM(\alpha)$ is totally positive
(see \cite[Theorem 4.4]{Pin10} for instance).
We refer the reader to \cite{Fla80,LMW16,LW-lcx,WZ16}.

\begin{thm}\label{tbc-sm}
Let $b,c>0$.
\begin{enumerate}[\rm (i)]
  \item The sequence $(T_n(b,c))_{n\ge 0}$ is log-convex if and only if $b^2\ge 2c$.
  \item The sequence $(T_n(b,c))_{n\ge 0}$ is SM if and only if $b^2\ge 4c$.
\end{enumerate}
\end{thm}
\begin{proof}
(i)\quad
Consider the tridiagonal matrix
\begin{equation}\label{j-tri}
  \JM=\left(
      \begin{array}{ccccc}
        b & 1 &  &  &  \\
        2c & b & 1 &  &  \\
         & c & b & 1 &  \\
        & & c & b & \ddots   \\
         &  &  & \ddots & \ddots \\
      \end{array}
    \right).
\end{equation}
It is known \cite[Theorem 2.3]{CLW15rm} that
if the matrix $\JM$ is TP$_2$,
then the sequence $(T_n(b,c))_{n\ge 0}$ is log-convex.
On the other hand,
the tridiagonal matrix $\JM$ is TP$_2$ if and only if $b^2\ge 2c$
\cite[Proposition 2.6 (i)]{CLW15ra}.
Thus $b^2\ge 2c$ implies that $(T_n(b,c))_{n\ge 0}$ is log-convex.

Conversely, if $(T_n(b,c))_{n\ge 0}$ is log-convex,
then in particular, we have $T_1(b,c)T_3(b,c)\ge T^2_2(b,c)$,
which is equivalent to $b^2\ge 2c$.

(ii)\quad
It is well known that a sequence $(a_k)_{k\ge 0}$ is SM if and only if there exists an integer $m\ge 0$ such that
$$\det[a_{i+j}]_{0\le i,j\le n}>0,\qquad \det[a_{i+j+1}]_{0\le i,j\le n}>0$$
for $0\le n<m$, and
$$\det[a_{i+j}]_{0\le i,j\le n}=\det[a_{i+j+1}]_{0\le i,j\le n}=0$$
for $n\ge m$ (see \cite[Theorem 1.3]{ST43} for instance).
Now
$\det[T_{i+j}(b,c)]_{0\le i,j\le n}=2^nc^{\binom{n+1}{2}}$
and $$\det[T_{i+j+1}(b,c)]_{0\le i,j\le n}=2^nc^{\binom{n+1}{2}}u_{n},$$
where $u_{n}$ is the $(n+1)$th principal leading minor of the matrix $\JM$
(see \cite[\S 8, Result 2]{Aig01} for instance).
It follows that $(T_n(b,c))_{n\ge 0}$ is SM if and only if $u_{n}$ are positive for all $n\ge 0$.
On the other hand,
an irreducible nonnegative tridiagonal matrix is TP
if and only if all its leading principal minors are positive
(see \cite[Example 2.2, p. 149]{Min88} for instance),
and the tridiagonal matrix $\JM$ is TP if and only if $b^2\ge 4c$
\cite[Proposition 2.6 (ii)]{CLW15ra}.
Thus $(T_n(b,c))_{n\ge 0}$ is SM if and only if $b^2\ge 4c$.
\end{proof}

\begin{rem}
When $b^2\ge 4c$, the sequence $(T_n(b,c))_{n\ge 0}$ is SM.
However, we do not know how to obtain the associated measure and whether it is discrete or continuous.

The central binomial coefficients $\binom{2n}{n}=T_n(2,1)$
and the central Delannoy numbers $D_n=T_n(3,2)$ are SM respectively.
See \cite{LMW16,WZ16,Zhu13} also.
\end{rem}

We have seen that if $b,c>0$ and $b^2\ge 2c$,
then $(T_n(b,c))_{n\ge 0}$ is log-convex,
i.e.,
$T_{i-1}(b,c)T_{j+1}(b,c)-T_{i}(b,c)T_{j}(b,c)$ are nonnegative for $j\ge i\ge 1$.
Actually, we have the following stronger result.

\begin{thm}\label{T-li}
Let $j\ge i\ge 1$. Then
$$T_{i-1}(b,c)T_{j+1}(b,c)-T_{i}(b,c)T_{j}(b,c)
=b^{\frac{1-(-1)^{i+j}}{2}}f_{ij}(b^2-2c,c),$$
where $f_{ij}(x,y)$ is a polynomial in $x$ and $y$ with nonnegative integer coefficients.
\end{thm}

To prove Theorem \ref{T-li},
we introduce some notations for convenience.
Let $\AM=[a_{ij}]_{i,j\ge 0}$ be a (finite or infinite) matrix.
We use
$
\AM\left(
\begin{array}{cc}
i_0,i_1 \\
j_0,j_1
\end{array}\right)
$
to denote the minor of order $2$ of $\AM$
determined by the rows indexed $i_1>i_0\ge 0$
and columns indexed $j_1>j_0\ge 0$.
Suppose that all entries of $\AM$ are polynomials in $b$ and $c$.
We say that $\AM$ is \n\ if there exist two polynomials $f(x,y)$ and $g(x,y)$ in $x$ and $y$ with nonnegative integer coefficients
such that
$$\AM\left(
\begin{array}{cc}
i_0,i_1 \\
j_0,j_1
\end{array}\right)
=
\left\{
  \begin{array}{rl}
   f(b^2-2c,c), & \hbox{if $i_0+i_1+j_0+j_1\equiv 0 \pmod 2$;} \\
   b\cdot g(b^2-2c,c), & \hbox{if $i_0+i_1+j_0+j_1\equiv 1 \pmod 2$.}
  \end{array}
\right.$$

\begin{lem}\label{lem-ab}
The product of an arbitrary number of \n\ matrices is still \n.
\end{lem}
\begin{proof}
By the associativity of matrix product,
it suffices to show that the product of two \n ~matrices is still \n.
Recall that the classical Cauchy-Binet formula
\begin{equation}\label{cb-f}
(\AM \BM)\binom{i_0,i_1}{j_0,j_1}=
\sum_{k_0<k_1}
\left[\AM\binom{i_0,i_1}{k_0,k_1}\cdot
\BM\binom{k_0,k_1}{j_0,j_1}\right].
\end{equation}
If $i_0+i_1+j_0+j_1\equiv 0 \pmod 2$,
then $i_0+i_1+k_0+k_1\equiv k_0+k_1+j_0+j_1 \pmod 2$.
Suppose that both $\AM$ and $\BM$ are \n.
Then each term
in the sum on the right hand side of \eqref{cb-f}
$$\AM\binom{i_0,i_1}{k_0,k_1}\cdot
\BM\binom{k_0,k_1}{j_0,j_1}=
\left\{
  \begin{array}{rl}
   f(b^2-2c,c), & \hbox{if $i_0+i_1+k_0+k_1\equiv 0 \pmod 2$;} \\
   b^2\cdot g(b^2-2c,c), & \hbox{if $i_0+i_1+k_0+k_1\equiv 1 \pmod 2$,}
  \end{array}
\right.$$
which is a polynomial in $b^2-2c$ and $c$ with nonnegative integer coefficients,
and so is the whole sum on the right hand side of \eqref{cb-f}.

Similarly,
if $i_0+i_1+j_0+j_1\equiv 1 \pmod 2$,
then each term in the sum on the right hand side of \eqref{cb-f} is form of $b\cdot g(b^2-2c,c)$,
and so is the whole sum on the left hand side of \eqref{cb-f}.
Thus we conclude that the product $\AM \BM$ is \n\ by the definition.
\end{proof}

\begin{lem}\label{lem-jcb}
The tridiagonal matrix $\JM$ is \n.
\end{lem}
\begin{proof}
The statement follows from the definition by checking nonzero minors of order $2$.
\end{proof}

\begin{lem}\label{lem-tcb}
The matrix $\Tbc$ is \n.
\end{lem}
\begin{proof}
Clearly, an infinite matrix is \n\ if and only if all leading principal submatrices are \n.
Let $\LM_n=[T_{n,k}(b,c)]_{0\le i,j\le n}$ be the $n$th leading principal submatrix of $\Tbc$.
Then it suffices to show that $\LM_n$ are \n\ for all $n\ge 0$.
We proceed by induction on $n$.
We have $|\LM_0|=|\LM_1|=1$.
Assume that $\LM_n$ is \n.
Denote
$$
\LM_n^*=\left[\begin{array}{cc}1 & O\\ O & \LM_n\\\end{array}\right]
$$
and
$$\JM^*_n
=\left[
  \begin{array}{c}
    e_n \\
    \JM_n \\
  \end{array}
\right],
$$
where $e_n=(1,0,\ldots,0)$ and $\JM_n$ is the $n$th leading principal submatrix of $\JM$.
Then
$$\JM^*_n
=
\left[
 \begin{array}{cccccc}
 1 & &  &  &  & \\
 b & 1 &  &  &  & \\
 2c & b & 1 &  &  & \\
 & c & b & \ddots &  & \\
 &  & \ddots & \ddots & 1 & \\
 &  &  & c & b & 1 \\
 \end{array}\right],
$$
and we can rewrite \eqref{tbc-d0} and \eqref{tbc-d1} as
$$\LM_{n+1}=\LM_n^*\cdot \JM^*_n.$$
By the assumption, $\LM_n$ is \n, and so $\LM^*_n$ is \n.
By Lemma \ref{lem-jcb}, $\JM$ is \n,
and so is the leading principal submatrix $\JM_n$,
as well as the matrix $\JM^*_n$.
The product $\LM_{n+1}$ is therefore \n.
Thus all leading principal submatrices of $\Tbc$ are \n, as required.
\end{proof}

We are now in a position to prove Theorem~\ref{T-li}.

\begin{proof}[Proof of Theorem~\ref{T-li}]
Let $\HM=[T_{n+k}(b,c)]_{n,k\ge 0}$ be the Hankel matrix of the Catalan-like numbers $(T_n(b,c))_{n\ge 0}$.
Then we need to show that the submatrix of $\HM$ determined by the first two rows is \n.
We do this by showing that the matrix $\HM$ is \n.

We have shown that
$T_n(b,c)$ are the Catalan-like numbers corresponding to
$\b=(b,b,b,\ldots)$ and $\c=(2c,c,c,\ldots)$.
Hence the Hankel matrix
\begin{equation}\label{tbc-snf}
  \HM=\Tbc\cdot \DM\cdot [\Tbc]^t
\end{equation}
by \eqref{snf}.
Clearly, the diagonal matrix $\DM=\mathrm{diag}(1,2c,2c^2.2c^3,\ldots)$ is \n.
By Lemma \ref{lem-tcb}, the matrix $\Tbc$ is \n, so is its transpose $[\Tbc]^t$.
Thus the product $\HM$ is \n\ by Lemma \ref{lem-ab}, as desired.
\end{proof}

\begin{coro}\label{tt-t2}
These exists a monic polynomial $f_n(x)$ of degree $n$ with nonnegative integer coefficients
such that
$$T_{n}(b,c)T_{n+2}(b,c)-T^2_{n+1}(b,c)=2c^{n+1}f_n\left(\frac{b^2-2c}{c}\right).$$
\end{coro}
\begin{proof}
Denote $\Delta_n=T_{n}(b,c)T_{n+2}(b,c)-T^2_{n+1}(b,c)$.
Then
$$
\Delta_n
=\HM\left(
    \begin{array}{cc}
      0, & 1 \\
      n, & n+1 \\
    \end{array}
  \right),
$$
where $\HM=[T_{n+k}(b,c)]_{n,k\ge 0}$.
By \eqref{tbc-snf} and the Cauchy-Binet formula,
$$
\Delta_n
=2c\cdot \Tbc\left(
    \begin{array}{cc}
      n, & n+1 \\
      0, & 1 \\
    \end{array}
  \right)
=2c\left[T_{n,0}(b,c)T_{n+1,1}(b,c)-T_{n,1}(b,c)T_{n+1,0}(b,c)\right].
$$
By \eqref{tnk}, we have
$$\Delta_n=2c\sum_{k=0}^na_kb^{2n-2k}c^k,$$
where all $a_k$ are integers.
In particular, $a_0=1$.
Thus
$$\Delta_n=2c\sum_{k=0}^na'_k(b^2-2c)^{n-k}c^k,$$
where all $a'_k$ are nonnegative integers by Theorem \ref{T-li},
and $a'_0=1$.
Define $f_n(x)=\sum_{k=0}^na'_kx^{n-k}$.
Then
$$\Delta_n=2c^{n+1}f_n\left(\frac{b^2-2c}{c}\right),$$
as desired.
\end{proof}

\section{Remarks}

The generalized Motzkin numbers $M_n(b,c)$ are closely related to the generalized central trinomial coefficients $T_n(b,c)$
and they possess many similar properties.
Recall that
$$
M_n(b,c)=\sum_{k=0}^{\lrf{n/2}}\frac{n!}{k!(k+1)!(n-2k)!}b^{n-2k}c^k
$$
and
$$
T_n(b,c)=\sum_{k=0}^{\lrf{n/2}}\frac{n!}{k!k!(n-2k)!}b^{n-2k}c^k.
$$
Then we have
\begin{enumerate}[\rm (i)]
  \item
  $\frac{\partial \left(cM_{n}(b,c)\right)}{\partial c}=T_n(b,c)$;
  \item
  $\frac{\partial T_{n}(b,c)}{\partial b}=nT_{n-1}(b,c)$ and $\frac{\partial M_{n}(b,c)}{\partial b}=nM_{n-1}(b,c)$;
  \item
  $\sum_{k=0}^n\binom{n}{k}M_k(b,c)a^{n-k}=M_n(a+b,c)$;
  \item
  $M_{i-1}(b,c)M_{j+1}(b,c)-M_{i}(b,c)M_{j}(b,c)=b^{\frac{1-(-1)^{i+j}}{2}}g_{ij}(b^2-c,c)$ for $j\ge i\ge 1$,
  where $g_{ij}(x,y)$ are polynomials in $x$ and $y$ with nonnegative integer coefficients;
  \item
  $M_{n}(b,c)M_{n+2}(b,c)-M^2_{n+1}(b,c)=c^{n+1}g_n\left(\frac{b^2-c}{c}\right)$,
  where $g_n(x)$ is a monic polynomial of degree $n$ with nonnegative integer coefficients.
\end{enumerate}

We omit the details for brevity.

\section*{Declaration of competing interest}

The authors declare no conflict of interest.

\section*{Acknowledgements}
The authors thank the anonymous referee
for valuable suggestions on improving the presentation of this paper.
This work was partially supported by the National Natural Science Foundation of China
(Grant Nos. 12001301 and 12171068).

\end{document}